\documentclass[12pt,a4paper]{article}
\usepackage{latexsym,amssymb,amsfonts,amsmath,amsthm,nccmath,float,enumitem,setspace,bm,authblk,graphicx,arydshln}
\usepackage[usenames,dvipsnames]{xcolor}
\usepackage[hidelinks]{hyperref}
\usepackage[margin=1.83cm]{geometry}
\usepackage{booktabs}

\hypersetup{
    colorlinks,
    linkcolor={blue!80!black},
    citecolor={blue!80!black},
    urlcolor={blue!80!black}
}

\newtheorem{Theorem}{Theorem}[section]

\newtheorem{Corollary}{Corollary}[section]

\newtheorem{Lemma}{Lemma}[section]

\def \leq {\leqslant}
\def \geq {\geqslant}
\def \Inv {{\rm Inv}}

\let\oldproofname=\proofname
\renewcommand{\proofname}{\rm\bf{\oldproofname}}

\numberwithin{equation}{section}

\allowdisplaybreaks

\begin{document}

\title{On the completion of $\epsilon$-dense partial Latin squares}

\author[a]{Shikang Yu}
\author[a,b]{Tao Feng \thanks{Supported by NSFC under Grant 12271023}}
\affil[a]{School of Mathematics and Statistics, Beijing Jiaotong University, Beijing, 100044, P.R. China}
\affil[b]{Hebei Provincial Key Laboratory of Mathematical Theory and Analysis for Network and Data Science, Beijing Jiaotong University, Beijing, 100044, P.R. China}
\renewcommand*{\Affilfont}{\small\it}
\renewcommand\Authands{ and }

\affil[ ]{healthyu@bjtu.edu.cn, tfeng@bjtu.edu.cn }
\date{}
\maketitle
\underline{}
\begin{abstract}
  A partial Latin square of order $n$ is called $\epsilon$-dense if each row and each column contains at most $\epsilon n$ filled cells, and each symbol occurs at most $\epsilon n$ times. A partial Latin square is said to be completable if its empty cells can be filled to obtain a Latin square. Daykin and H\"{a}ggkvist conjectured that every $\frac{1}{4}$-dense partial Latin square is completable. In this paper, we show that for all sufficiently large integers $n$, every $\frac{2}{25}$-dense partial Latin square of order $n$ is completable. The proof is obtained by establishing that there exists an $\eta > 0$ such that every triangle-divisible balanced tripartite graph on $3n$ vertices with partite minimum degree at least $(\frac{23}{25}-\eta)n$ admits a fractional triangle decomposition.
\end{abstract}


\section{Introduction}\label{sec:introduction}

A \textit{partial Latin square} of order $n$ is an $n\times n$ array such that each cell is either empty or filled with one of $n$ symbols, with no symbol being repeated in any row or column. A partial Latin square in which every cell is filled is a \textit{Latin square}. A partial Latin square is called \textit{completable} if there is some way to fill in all of its empty cells to obtain a Latin square. The problem of characterizing which partial Latin squares are completable is a classical problem in combinatorics and has received extensive attention. Hall \cite{Hall} proved that any partial Latin square of order $n$ in which the first $k$ rows are completely filled and the remaining rows are empty is completable. Ryser \cite{Ryser} established that any partial Latin square of order $n$ whose non-empty cells all lie within a set of $s$ rows and $t$ columns, with $s + t \leq n$, is completable. These studies impose rather strong conditions on the positions of the non-empty cells. In contrast, restricting only the total number of filled cells and imposing no requirements on their positions, Evans \cite{Evans} conjectured that any partial Latin square of order $n$ with at most $n-1$ non-empty cells is completable. This conjecture was later proved independently by Smetianuk \cite{Smetianuk} and by Anderson and Hilton \cite{AH}.

This paper focuses on the problem of determining whether a partial Latin square is completable, given restrictions on both the positions and the number of its filled cells. Let $0 \leq \epsilon \leq 1$. A partial Latin square of order $n$ is called \textit{$\epsilon$-dense} if each row and each column contains at most $\epsilon n$ non-empty cells, and each symbol occurs at most $\epsilon n$ times in the entire array. Daykin and H\"{a}ggkvist \cite{DH} conjectured that every $1/4$-dense partial Latin square is completable. Wanless \cite{Wanless} pointed out that, for any sufficiently small $c>0$, there exists a $(1/4+c)$-dense partial Latin square that is not completable, proving that the constant $1/4$ is tight. Chetwynd and H\"{a}ggkvist \cite{CH} demonstrated that, for any sufficiently large even integer $n$, every $10^{-5}$-dense partial Latin square of order $n$ is completable. Gustavsson \cite{Gustavsson} showed that for all sufficiently large $n$, every $10^{-7}$-dense partial Latin square of order $n$ is completable. Bartlett \cite{Bartlett} established that, for sufficiently large $n$, every $10^{-4}$-dense partial Latin square of order $n$ is completable. The previously best-known result was due to Bowditch and Dukes \cite{BD}, who proved that for sufficiently large $n$, every $1/25$-dense partial Latin square of order $n$ is completable. In this paper, we establish the following result.

\begin{Theorem}\label{thm:main1}
For all sufficiently large integers $n$, every $\frac{2}{25}$-dense partial Latin square of order $n$ is completable.
\end{Theorem}

We use graph-theoretic models to state and prove Theorem \ref{thm:main1}. All graphs considered here are finite, undirected, and simple. Let $H$ be a graph. Let $V(H)$ and $E(H)$ denote the vertex set and the edge set of $H$, respectively. For $v \in V(H)$ and $U \subseteq V(H)$, let $N(v, U)$ denote the set of neighbors of $v$ in $U$, and let $d(v, U)=|N(v, U)|$.

A \textit{triangle decomposition} of a graph is a partition of its edge set into edge-disjoint triangles. A tripartite graph is called \textit{balanced} if each of its partite sets has the same size. Let $X$, $Y$, and $Z$ denote the row, column, and symbol sets, respectively, with $|X|=|Y|=|Z|=n$. A partial Latin square $P$ can be formally identified with its set of constituent ordered triples: $P = \{ (x, y, z) \in X \times Y \times Z : \text{symbol } z \text{ occupies row } x \text{ and column } y \}$. Each triple $(x, y, z) \in P$ naturally induces a triangle in the complete balanced tripartite graph $K_{X,Y,Z}$. In this sense, a Latin square of order $n$ is equivalent to a triangle decomposition of $K_{X,Y,Z}$, and a partial Latin square $P$ of order $n$ is completable if and only if there exists a triangle decomposition of the spanning subgraph $\overline{P}$ of $K_{X,Y,Z}$ with edge set $E(K_{X,Y,Z}) \setminus E_P$, where $E_P$ denotes the union of the edge sets of the triangles induced by the triples in $P$.

Let $G$ be a balanced tripartite graph with partite sets $X$, $Y$, and $Z$. If $G$ admits a triangle decomposition, then
$$d(x, Y) = d(x, Z), \quad d(y, X) = d(y, Z), \quad \text{and} \quad d(z, X) = d(z, Y)$$
for every $x \in X$, $y \in Y$, and $z \in Z$. We call such a balanced tripartite graph $G$ \textit{triangle-divisible}. The \textit{partite minimum degree} of $G$ is defined as
$$\hat{\delta}(G) := \min\{d(v, W) : v \in X \cup Y \cup Z, \ W \in \{X, Y, Z\}, \ v \notin W\}.$$
Note that if $P$ is an $\epsilon$-dense partial Latin square of order $n$, then the graph $\overline{P}$ is triangle-divisible and satisfies $\hat{\delta}(\overline{P}) \geq (1-\epsilon)n$. Hence, to prove Theorem \ref{thm:main1}, it suffices to prove the following theorem.

\begin{Theorem}\label{thm:main2}
For all sufficiently large integers $n$, every triangle-divisible balanced tripartite graph on $3n$ vertices with partite minimum degree at least $\frac{23}{25}n$ admits a triangle decomposition.
\end{Theorem}

Note that an $\epsilon$-dense partial Latin square of order $n$ corresponds to a triangle-divisible balanced tripartite graph on $3n$ vertices with partite minimum degree at least $(1-\epsilon)n$, but the converse does not necessarily hold: not every such graph is the complement of a set of edge-disjoint triangles in a complete balanced tripartite graph. Hence, Theorem~\ref{thm:main2} is strictly stronger than Theorem~\ref{thm:main1}.

A \textit{fractional triangle decomposition} of $G$ is an assignment of nonnegative real weights to the triangles of $G$ such that, for each edge of $G$, the sum of the weights of all triangles containing that edge is exactly $1$. A triangle decomposition can thus be viewed as a fractional triangle decomposition in which each assigned weight is either $0$ or $1$. Barber, K\"{u}hn, Lo, Osthus, and Taylor \cite[Corollary 1.6]{BKLOT} employed the iterative absorption approach, together with a result from \cite{HR}, to establish that for sufficiently large $n$, if every triangle-divisible balanced tripartite graph on $3n$ vertices with partite minimum degree at least $cn$ admits a fractional triangle decomposition, then every such graph with partite minimum degree at least $c'n$ admits a triangle decomposition, provided that $c' > c \geq 3/4$. Therefore, to prove Theorem \ref{thm:main2}, it suffices to prove the following theorem.

\begin{Theorem}\label{thm:main3}
There exists an $\eta > 0$ such that every triangle-divisible balanced tripartite graph on $3n$ vertices with partite minimum degree at least $(\frac{23}{25}-\eta)n$ admits a fractional triangle decomposition.
\end{Theorem}

\section{Proof of Theorem~\ref{thm:main3}}\label{sec:outline}

First, we fix some notation and basic facts that will be used frequently later. 

Let $G$ be a triangle-divisible balanced tripartite graph with partite sets $X, Y$, and $Z$ satisfying $$\hat{\delta}(G) \geq (1-\gamma)n,$$ where $|X|=|Y|=|Z|=n$ and $0 \leq \gamma \leq 1/4$. For disjoint subsets $U_1,U_2\subseteq V(G)$, let $E(U_1,U_2)$ denote the set of edges in $G$ with one vertex in $U_1$ and the other in $U_2$. Let $\mathcal{T}(G)$ denote the set of all triangles in $G$. For any $e\in E(G)$, let $$t_e := | \{ T \in \mathcal{T}(G) : e\in T \} |.$$ Observe that $t_e$ is exactly the number of common neighbors of the endpoints of $e$ in the third part. Combining the degree condition with the inclusion-exclusion principle yields $$(1-2\gamma)n \leq t_e \leq n.$$ Since $G$ is triangle-divisible, $|E(X, Y)| = |E(X, Z)| = |E(Y, Z)|$. Throughout, we use $x$, $y$, and $z$ to denote vertices in $X$, $Y$, and $Z$, respectively.

Let $\mathbb{R}^{E(G)}$ and $\mathbb{R}^{\mathcal{T}(G)}$ be the real vector spaces whose coordinates are indexed by $E(G)$ and $\mathcal{T}(G)$, respectively. For any $\mathbf{v} \in \mathbb{R}^{\mathcal{T}(G)}$, we write $\mathbf{v} \geq 0$ to indicate that all entries of $\mathbf{v}$ are non-negative. Let $\mathbf{1} \in \mathbb{R}^{E(G)}$ denote the all-ones vector. Let $A$ be the matrix whose rows are indexed by edges $e \in E(G)$ and columns by triangles $T \in \mathcal{T}(G)$, where the $(e,T)$-entry is $1$ if $e \in T$ and $0$ otherwise.

Define the vector $\mathbf{a}\in \mathbb{R}^{\mathcal{T}(G)}$ as follows: for each triangle $T = xyz \in \mathcal{T}(G)$,
$$ \mathbf{a}_T := \frac{1}{3} \left( \frac{1}{t_{xy}} + \frac{1}{t_{xz}} + \frac{1}{t_{yz}} \right). $$
Let $\mathbf{b}: = A \mathbf{a}  \in \mathbb{R}^{E(G)}$, and let $\mathbf{b}^\star: = \mathbf{1} - \mathbf{b} \in \mathbb{R}^{E(G)}$. We then have the following lemma.

\begin{Lemma}\label{lem:frac1}
If there exists a vector $\mathbf{a}^\star \in \mathbb{R}^{\mathcal{T}(G)}$ such that $\mathbf{a} + \mathbf{a}^\star \geq 0 $ and $A \mathbf{a}^\star  = \mathbf{b}^\star$, then the graph $G$ admits a fractional triangle decomposition.
\end{Lemma}

\begin{proof}
Since $A(\mathbf{a} + \mathbf{a}^\star) = A \mathbf{a} + A \mathbf{a}^\star = \mathbf{b} + \mathbf{1} - \mathbf{b} = \mathbf{1} \in \mathbb{R}^{E(G)}$, the weight function $\omega : \mathcal{T}(G) \to \mathbb{R}$ defined by $\omega(T) = (\mathbf{a} + \mathbf{a}^\star)_T$ for each $T \in \mathcal{T}(G)$ yields a fractional triangle decomposition of $G$.
\end{proof}

The remainder of this paper is devoted to proving that, under the condition $\gamma = 2/25+\eta$ for some $\eta>0$, there exists a vector $\mathbf{a}^\star \in \mathbb{R}^{\mathcal{T}(G)}$ satisfying the conditions of Lemma~\ref{lem:frac1}.

\subsection{Good vectors in $\mathbb{R}^{E(G)}$}\label{subsec:repairable}

Since $G$ is triangle-divisible, we have the degree equalities $d(x, Y) = d(x, Z)$, $d(y, X) = d(y, Z)$,  and $d(z, X) = d(z, Y)$ for all $x \in X$, $y \in Y$, and $z \in Z$. This motivates us to introduce a special class of vectors in $\mathbb{R}^{E(G)}$, called good vectors, which will be used in Section~\ref{subsec:Iterative} to construct a vector $\mathbf{a}^\star \in \mathbb{R}^{\mathcal{T}(G)}$ satisfying the conditions in Lemma~\ref{lem:frac1}.

For a vector $\mathbf{c} \in \mathbb{R}^{E(G)}$, by abuse of notation,  we regard $\mathbf{c}$ as a vector in $\mathbb{R}^{E(K_{X,Y,Z})}$ whenever necessary, by taking its entries to be 0 outside $E(G)$.

We say that $\mathbf{c}\in \mathbb{R}^{E(G)}$ is \textit{good} if it satisfies the following equations:
$$ \sum_{y \in Y} \mathbf{c}_{xy} = \sum_{z \in Z} \mathbf{c}_{xz} \quad \text{for all } x \in X, $$
$$ \sum_{x \in X} \mathbf{c}_{xy} = \sum_{z \in Z} \mathbf{c}_{yz} \quad \text{for all } y \in Y, $$
and
$$ \sum_{x \in X} \mathbf{c}_{xz} = \sum_{y \in Y} \mathbf{c}_{yz} \quad \text{for all } z \in Z. $$

\begin{Lemma}\label{lem:repairable1}
For any vector $\mathbf{v} \in \mathbb{R}^{\mathcal{T}(G)}$, the vector $A\mathbf{v} \in \mathbb{R}^{E(G)}$ is good.
\end{Lemma}

\begin{proof}
Without loss of generality, fix an arbitrary vertex $x\in X$. Then
$$ \sum_{y \in Y} (A\mathbf{v})_{xy} = \sum_{y \in Y} \sum_{T \in \mathcal{T}(G) : xy \in T} \mathbf{v}_T = \sum_{T \in \mathcal{T}(G) : x \in V(T)} \mathbf{v}_T = \sum_{z \in Z} \sum_{T \in \mathcal{T}(G) : xz \in T} \mathbf{v}_T = \sum_{z \in Z} (A\mathbf{v})_{xz}. $$
Thus, the vector $A\mathbf{v}$ is good.
\end{proof}

\begin{Lemma}\label{lem:repairable2}
If $\mathbf{c}_1, \mathbf{c}_2 \in \mathbb{R}^{E(G)}$ are good, then $\mathbf{c}_1 - \mathbf{c}_2$ is also good.
\end{Lemma}

\begin{proof}
Without loss of generality, fix an arbitrary vertex $x\in X$. Since $\mathbf{c}_1$ and $\mathbf{c}_2$ are good, we have
$$ \sum_{y \in Y} (\mathbf{c}_1 - \mathbf{c}_2)_{xy} = \sum_{y \in Y} (\mathbf{c}_1)_{xy} - \sum_{y \in Y} (\mathbf{c}_2)_{xy} = \sum_{z \in Z} (\mathbf{c}_1)_{xz} - \sum_{z \in Z} (\mathbf{c}_2)_{xz} = \sum_{z \in Z} (\mathbf{c}_1 - \mathbf{c}_2)_{xz}. $$
Thus, $\mathbf{c}_1 - \mathbf{c}_2$ is good.
\end{proof}

\begin{Lemma}\label{lem:repairable3}
The vector $\mathbf{b}^\star \in \mathbb{R}^{E(G)}$ is good.
\end{Lemma}

\begin{proof}
Recall that $\mathbf{b}^\star = \mathbf{1} - \mathbf{b}$. By Lemma \ref{lem:repairable1}, $\mathbf{b} = A \mathbf{a}$ is good. Thus, by Lemma \ref{lem:repairable2}, it suffices to show that the all-ones vector $\mathbf{1} \in \mathbb{R}^{E(G)}$ is good. Without loss of generality, fix an arbitrary vertex $x \in X$. Since $G$ is triangle-divisible, we have
$$ \sum_{y \in Y} \mathbf{1}_{xy} = d(x, Y) = d(x, Z) = \sum_{z \in Z} \mathbf{1}_{xz},$$
which yields that $\mathbf{1} \in \mathbb{R}^{E(G)}$ is good.
\end{proof}

\subsection{Constructing $\mathbf{a}^\star$ from the good vector $\mathbf{b}^\star$} \label{subsec:Iterative}

By Lemma \ref{lem:repairable3}, the vector $\mathbf{b}^\star$ is good. In this section, we construct a vector $\mathbf{a}^\star \in \mathbb{R}^{\mathcal{T}(G)}$, based on $\mathbf{b}^\star$, such that $\mathbf{a}^\star$ satisfies the conditions in Lemma~\ref{lem:frac1}. Before that, we need to define a vector $\widetilde{\Inv}(\mathbf{c}) \in \mathbb{R}^{\mathcal{T}(G)}$ for each good vector $\mathbf{c} \in \mathbb{R}^{E(G)}$.

For a good vector $\mathbf{c} \in \mathbb{R}^{E(G)}$, it immediately follows that
$$ \sum_{x \in X} \sum_{y \in Y} \mathbf{c}_{xy} =  \sum_{x \in X} \sum_{z \in Z} \mathbf{c}_{xz} = \sum_{y \in Y} \sum_{z \in Z} \mathbf{c}_{yz}. $$
Let
$$ S_\mathbf{c}(x) := \sum_{y \in Y} \mathbf{c}_{xy} = \sum_{z \in Z} \mathbf{c}_{xz} \quad \text{for all } x \in X, $$
$$ S_\mathbf{c}(y) := \sum_{x \in X} \mathbf{c}_{xy} = \sum_{z \in Z} \mathbf{c}_{yz} \quad \text{for all } y \in Y, $$
$$ S_\mathbf{c}(z) := \sum_{x \in X} \mathbf{c}_{xz} = \sum_{y \in Y} \mathbf{c}_{yz} \quad \text{for all } z \in Z, $$
and let
$$ S(\mathbf{c}) :=  \sum_{x \in X} \sum_{y \in Y} \mathbf{c}_{xy} =  \sum_{x \in X} \sum_{z \in Z} \mathbf{c}_{xz} = \sum_{y \in Y} \sum_{z \in Z} \mathbf{c}_{yz}. $$
Define a vector $\Inv(\mathbf{c}) \in \mathbb{R}^{X \times Y \times Z}$, where $\mathbb{R}^{X \times Y \times Z}$ is the real vector space with coordinates indexed by $X \times Y \times Z$, by
$$ \Inv(\mathbf{c})_{xyz} := \frac{\mathbf{c}_{xy} + \mathbf{c}_{xz} + \mathbf{c}_{yz}}{n} - \frac{S_\mathbf{c}(x) + S_\mathbf{c}(y) + S_\mathbf{c}(z)}{n^2} + \frac{S(\mathbf{c})}{n^3}.$$
We further define a vector $\widetilde{\Inv}(\mathbf{c}) \in \mathbb{R}^{\mathcal{T}(G)}$ by
$$ \widetilde{\Inv}(\mathbf{c})_T := \Inv(\mathbf{c})_{xyz} \quad \text{for each } T = xyz \in \mathcal{T}(G). $$

The following lemma characterizes the properties of $\Inv(\mathbf{c})$, and will be used frequently in Section \ref{sec:estimation}.

\begin{Lemma}\label{lem:Inv}
If $\mathbf{c} \in \mathbb{R}^{E(G)}$ is good, then
$$ \sum_{z \in Z } \Inv(\mathbf{c})_{xyz} = \mathbf{c}_{xy} \quad \text{for all } xy \in X\times Y, $$
$$ \sum_{y \in Y } \Inv(\mathbf{c})_{xyz} = \mathbf{c}_{xz} \quad \text{for all } xz \in X\times Z, $$
and
$$ \sum_{x \in X } \Inv(\mathbf{c})_{xyz} = \mathbf{c}_{yz} \quad \text{for all } yz \in Y\times Z. $$
\end{Lemma}

\begin{proof}
Without loss of generality, fix an arbitrary pair $xy \in X \times Y$. It suffices to prove that $\sum_{z \in Z } \Inv(\mathbf{c})_{xyz} = \mathbf{c}_{xy}$. Since $\mathbf{c}$ is good, we have
$$ \sum_{z \in Z} \mathbf{c}_{xz} = S_\mathbf{c}(x), \quad \sum_{z \in Z} \mathbf{c}_{yz} = S_\mathbf{c}(y), \quad  \text{and} \quad  \sum_{z \in Z} S_\mathbf{c}(z) = S(\mathbf{c}). $$
Therefore,
\begin{align*}
\sum_{z \in Z} \Inv(\mathbf{c})_{xyz}
&= \sum_{z \in Z} \left(  \frac{\mathbf{c}_{xy} + \mathbf{c}_{xz} + \mathbf{c}_{yz}}{n} - \frac{S_\mathbf{c}(x) + S_\mathbf{c}(y) + S_\mathbf{c}(z)}{n^2} + \frac{S(\mathbf{c})}{n^3} \right) \\
&= \frac{1}{n} \left( \sum_{z \in Z} \mathbf{c}_{xy} + \sum_{z \in Z} \mathbf{c}_{xz} + \sum_{z \in Z} \mathbf{c}_{yz} \right) \\
&\quad - \frac{1}{n^2} \left( \sum_{z \in Z} S_\mathbf{c}(x) + \sum_{z \in Z} S_\mathbf{c}(y) + \sum_{z \in Z} S_\mathbf{c}(z) \right) + \frac{1}{n^3} \sum_{z \in Z} S(\mathbf{c}) \\
&= \frac{1}{n} \left( n \mathbf{c}_{xy} + S_\mathbf{c}(x) + S_\mathbf{c}(y) \right) - \frac{1}{n^2} \left( n S_\mathbf{c}(x) + n S_\mathbf{c}(y) + S(\mathbf{c}) \right) + \frac{S(\mathbf{c})}{n^2} = \mathbf{c}_{xy}.
\end{align*}
This completes the proof.
\end{proof}

Starting with the good vector $\mathbf{b}^{(0)} := \mathbf{b}^\star \in \mathbb{R}^{E(G)}$, we recursively define a sequence $\{ \mathbf{b}^{(k)} \}_{k=0}^{\infty}$ of good vectors in $\mathbb{R}^{E(G)}$. For each $k \ge 0$, we set
$$ \mathbf{b}^{(k+1)} := \mathbf{b}^{(k)} - A \widetilde{\Inv}(\mathbf{b}^{(k)}). $$
It follows from Lemmas \ref{lem:repairable1}, \ref{lem:repairable2}, and \ref{lem:repairable3} that for all $k\geq 0$, $\mathbf{b}^{(k)}$ is good, and hence $\Inv(\mathbf{b}^{(k)})$ is well-defined. We then have the following lemma.

\begin{Lemma}\label{lem:frac2}
Suppose that the series $ \sum_{k=0}^{\infty} \widetilde{\Inv}(\mathbf{b}^{(k)})$ converges absolutely component-wise. Let
$$\mathbf{a}^\star = \sum_{k=0}^{\infty} \widetilde{\Inv}(\mathbf{b}^{(k)}) \in \mathbb{R}^{\mathcal{T}(G)}.$$
If $\max \left\{ \left| \mathbf{a}^\star_T \right| : T \in \mathcal{T}(G) \right\} \leq \frac{1}{n}$ and $A \mathbf{a}^\star = \mathbf{b}^\star$, then the graph $G$ admits a fractional triangle decomposition.
\end{Lemma}

\begin{proof}
Since $t_e \leq n$ for any $e \in E(G)$, we have that for any $T \in \mathcal{T}(G)$,
$$ \mathbf{a}_T \geq \frac{1}{3} \left( \frac{1}{n} + \frac{1}{n} + \frac{1}{n} \right) = \frac{1}{n}. $$
By assumption, $\max \left\{ \left| \mathbf{a}^\star_T \right| : T \in \mathcal{T}(G) \right\} \leq \frac{1}{n}$. Thus $\mathbf{a} + \mathbf{a}^\star \geq 0$.
Furthermore, since $A \mathbf{a}^\star = \mathbf{b}^\star$, by Lemma \ref{lem:frac1}, $G$ admits a fractional triangle decomposition.
\end{proof}

\subsection{The choice of $\gamma$}\label{sec:proof}

Now, to prove Theorem \ref{thm:main3}, it suffices to choose appropriate $\gamma$ with $0 \leq \gamma \leq 1/4$ such that the series $ \sum_{k=0}^{\infty} \widetilde{\Inv}(\mathbf{b}^{(k)}) $ satisfies the conditions in Lemma \ref{lem:frac2}.

For a good vector $\mathbf{c} \in \mathbb{R}^{E(G)}$, let
$$ M_1(\mathbf{c}) := \max_{e\in E(G)} |\mathbf{c}_e|, \quad M_2(\mathbf{c}) := \frac{1}{n} \max_{v \in X \cup Y \cup Z} |S_\mathbf{c}(v)|, \quad \text{and} \quad M_3(\mathbf{c}) := \frac{|S(\mathbf{c})|}{n^2}.$$
Furthermore, let
$$M(\mathbf{c}) := 2M_1(\mathbf{c}) + 3M_2(\mathbf{c}) + M_3(\mathbf{c}).$$

The following Lemmas \ref{lem:MMbound} and \ref{lem:estimation} will be employed to prove Lemma \ref{lem:converges}. For readability, their proofs are given in Section \ref{sec:MMbound} and Section \ref{sec:estimation}, respectively.

\begin{Lemma}\label{lem:MMbound}
For $0 \leq \gamma \leq 1/4$, the vector $\mathbf{b}^\star \in \mathbb{R}^{E(G)}$ satisfies $3M_1(\mathbf{b}^\star) + 3M_2(\mathbf{b}^\star) + M_3(\mathbf{b}^\star) \leq \frac{4\gamma}{1-2\gamma}$.
\end{Lemma}

\begin{Lemma}\label{lem:estimation}
For $0 \leq \gamma \leq 1/4$ and the sequence $\{\mathbf{b}^{(k)}\}_{k=0}^{\infty}$ defined before Lemma $\ref{lem:frac2}$, 
\begin{itemize}
    \item[$(1)$] $ M(\mathbf{b}^{(k)}) \leq (7\gamma + 6\gamma^2)^k \frac{10\gamma}{3(1-2\gamma)} $;
    \item[$(2)$] $ 3M_1(\mathbf{b}^{(k+1)}) + 3M_2(\mathbf{b}^{(k+1)}) + M_3(\mathbf{b}^{(k+1)}) \leq (9\gamma + 6\gamma^2)(7\gamma + 6\gamma^2)^k \frac{10\gamma}{3(1-2\gamma)} $.
\end{itemize}
\end{Lemma}

\begin{Lemma}\label{lem:converges}
If the nonnegative real $\gamma$ satisfies
$$7\gamma + 6\gamma^2 < 1 \quad\text{and}\quad
\frac{2\gamma}{3(1-2\gamma)} \left( 6 + \frac{5(9\gamma + 6\gamma^2)}{1 - 7\gamma - 6\gamma^2} \right) \le 1,$$
then the series $\mathbf{a}^\star = \sum_{k=0}^{\infty} \widetilde{\Inv}(\mathbf{b}^{(k)})$ converges absolutely component-wise. Moreover, $\max \{ |\mathbf{a}^\star_T| : T \in \mathcal{T}(G) \} \le \frac{1}{n}$ and $A\mathbf{a}^\star = \mathbf{b}^\star$.
\end{Lemma}

\begin{proof}
Without loss of generality, fix an arbitrary triangle $T = xyz \in \mathcal{T}(G)$. We first prove that the series $\sum_{k=0}^{\infty} | \widetilde{\Inv}(\mathbf{b}^{(k)})_T | \leq \frac{1}{n}$, which implies that the series $\sum_{k=0}^{\infty} \widetilde{\Inv}(\mathbf{b}^{(k)})_T$ converges absolutely, and $|\mathbf{a}^\star_T|\leq \sum_{k=0}^{\infty} | \widetilde{\Inv}(\mathbf{b}^{(k)})_T | \leq \frac{1}{n}$.

For any $k \geq 0$, the definition of $\Inv(\cdot)$ implies
$$ | \widetilde{\Inv}(\mathbf{b}^{(k)})_T | \leq \frac{3M_1(\mathbf{b}^{(k)}) + 3M_2(\mathbf{b}^{(k)}) + M_3(\mathbf{b}^{(k)})}{n}. $$
Recall that $\mathbf{b}^{(0)} = \mathbf{b}^\star$. For $k=0$, Lemma \ref{lem:MMbound} yields
$$ | \widetilde{\Inv}(\mathbf{b}^{(0)})_T | \leq \frac{4\gamma}{(1-2\gamma)n}. $$
For all $k \geq 1$, by Lemma \ref{lem:estimation}$(2)$, we have
$$ | \widetilde{\Inv}(\mathbf{b}^{(k)})_T | \leq (9\gamma + 6\gamma^2)(7\gamma + 6\gamma^2)^{k-1} \frac{10\gamma}{3(1-2\gamma)n}. $$
Hence, we have
\begin{align*}
\sum_{k=0}^{\infty} | \widetilde{\Inv}(\mathbf{b}^{(k)})_T |
&\leq \frac{4\gamma}{(1-2\gamma)n} + \frac{10\gamma(9\gamma + 6\gamma^2)}{3(1-2\gamma)n} \sum_{k=1}^{\infty} (7\gamma + 6\gamma^2)^{k-1} \\
&= \frac{2\gamma}{3(1-2\gamma)} \left( 6 + \frac{5(9\gamma + 6\gamma^2)}{1 - 7\gamma - 6\gamma^2} \right) \cdot \frac{1}{n} \leq \frac{1}{n}.
\end{align*}

Next, we prove that $A\mathbf{a}^\star = \mathbf{b}^\star$. By the recursive definition of $\mathbf{b}^{(k)}$, we have
\begin{align*}
    \mathbf{b}^{(1)} &= \mathbf{b}^{(0)} - A \widetilde{\Inv}(\mathbf{b}^{(0)}), \\
    \mathbf{b}^{(2)} &= \mathbf{b}^{(1)} - A \widetilde{\Inv}(\mathbf{b}^{(1)}), \\
    &\ \ \ \ \ \ \ \ \ \ \vdots \\
    \mathbf{b}^{(h)} &= \mathbf{b}^{(h-1)} - A \widetilde{\Inv}(\mathbf{b}^{(h-1)}).
\end{align*}
Summing these $h$ equations yields
$$ \mathbf{b}^{(h)} = \mathbf{b}^{(0)} - \sum_{k=0}^{h-1} A \widetilde{\Inv}(\mathbf{b}^{(k)}) = \mathbf{b}^{(0)} - A \left( \sum_{k=0}^{h-1} \widetilde{\Inv}(\mathbf{b}^{(k)}) \right). $$
By Lemma \ref{lem:estimation}$(1)$, we have
$$ M(\mathbf{b}^{(h)}) \leq (7\gamma + 6\gamma^2)^h \frac{10\gamma}{3(1-2\gamma)}. $$
Since $0\leq 7\gamma + 6\gamma^2 < 1$, as $h \to \infty$, we have $M(\mathbf{b}^{(h)}) \to 0$. This forces $M_1(\mathbf{b}^{(h)}) \to 0$, which implies $\mathbf{b}^{(h)} \to \mathbf{0}$. Therefore, taking the limit as $h \to \infty$ gives
$$ \mathbf{0} = \mathbf{b}^{(0)} - A \left( \sum_{k=0}^{\infty} \widetilde{\Inv}(\mathbf{b}^{(k)}) \right) = \mathbf{b}^\star - A \mathbf{a}^\star. $$
Thus, $A \mathbf{a}^\star = \mathbf{b}^\star$. This completes the proof.
\end{proof}

Now we are in a position to give the proof of Theorem \ref{thm:main3}.

\begin{proof}[\bf Proof of Theorem \ref{thm:main3}]
It is readily checked that the two inequalities for $\gamma$ in Lemma \ref{lem:converges} hold for any $0\leq \gamma\leq 0.0805$, and so we can apply Lemma \ref{lem:converges} and Lemma \ref{lem:frac2} to complete the proof.
\end{proof}

\section{Proof of Lemma \ref{lem:MMbound}}\label{sec:MMbound}

Before giving the proof of Lemma \ref{lem:MMbound}, we need several lemmas. Recall that, for $v \in V(G)$ and $U \subseteq V(G)$, $N(v, U)$ denotes the set of neighbors of $v$ in $U$, and $d(v, U)$ is defined as $|N(v, U)|$.

\begin{Lemma}\label{lem:M1bound}
For $0 \leq \gamma \leq 1/4$, the vector $\mathbf{b}^\star \in \mathbb{R}^{E(G)}$ satisfies $ M_1(\mathbf{b}^\star) \leq \frac{2\gamma}{3(1 - 2\gamma)}$.
\end{Lemma}

\begin{proof}
Without loss of generality, fix an arbitrary edge $e = xy \in E(X,Y)$. It suffices to show that $|\mathbf{b}^\star_e| = |1-\mathbf{b}_e| \leq \frac{2\gamma}{3(1 - 2\gamma)}$. Define
$$ \alpha_{xy} := \sum_{z : xyz \in \mathcal{T}(G)} \frac{1}{t_{xz}} \quad \text{and} \quad \beta_{xy} := \sum_{z : xyz \in \mathcal{T}(G)} \frac{1}{t_{yz}}.$$
Recall that $\mathbf{b}= A \mathbf{a}$. Then
\begin{align*}
\mathbf{b}_e &= \sum_{T\in \mathcal{T}(G) : e\in T } \mathbf{a}_T =  \sum_{z : xyz \in \mathcal{T}(G)} \mathbf{a}_{xyz} = \sum_{z : xyz \in \mathcal{T}(G)} \frac{1}{3} \left( \frac{1}{t_{xy}} + \frac{1}{t_{xz}} + \frac{1}{t_{yz}} \right) \\
&= \frac{1}{3} \left( \sum_{z : xyz \in \mathcal{T}(G)} \frac{1}{t_{xy}} \right) + \frac{1}{3} \left( \sum_{z : xyz \in \mathcal{T}(G)} \frac{1}{t_{xz}} \right) + \frac{1}{3} \left( \sum_{z : xyz \in \mathcal{T}(G)} \frac{1}{t_{yz}} \right) = \frac{1}{3} + \frac{1}{3} \alpha_{xy} + \frac{1}{3} \beta_{xy}.
\end{align*}
Since $G$ is triangle-divisible, write 
$$d_x := d(x,Y)=d(x,Z).$$ 
For any $z \in Z$, since $\hat{\delta}(G) \geq (1 - \gamma)n$, we have $d(z,Y) \geq (1 - \gamma)n$. Thus, by the inclusion-exclusion principle, we get $d_x - \gamma n \leq t_{xz} = |N(x,Y) \cap N(z,Y)| \leq d_x$. Similarly, we have $d_x - \gamma n \leq t_{xy} \leq d_x$. Hence,
\begin{align*}
\alpha_{xy} &= \sum_{z : xyz \in \mathcal{T}(G)} \frac{1}{t_{xz}} \leq \sum_{z : xyz \in \mathcal{T}(G)} \frac{1}{d_x - \gamma n} = \frac{t_{xy}}{d_x - \gamma n} 
\leq \frac{d_x}{d_x - \gamma n} = \frac{1}{1-\frac{\gamma n}{d_x}} \leq \frac{1}{1-\frac{\gamma n}{(1-\gamma) n}} = \frac{1 - \gamma}{1 - 2\gamma},
\end{align*}
and
\begin{align*}
\alpha_{xy} &= \sum_{z : xyz \in \mathcal{T}(G)} \frac{1}{t_{xz}} \geq \sum_{z : xyz \in \mathcal{T}(G)} \frac{1}{d_x} = \frac{t_{xy}}{d_x} 
\geq \frac{d_x - \gamma n}{d_x} = 1-\frac{\gamma n}{d_x} \geq 1-\frac{\gamma n}{(1-\gamma) n} = \frac{1 - 2\gamma}{1 - \gamma}.
\end{align*}
That is,
$$ \frac{1 - 2\gamma}{1 - \gamma} \leq \alpha_{xy} \leq \frac{1 - \gamma}{1 - 2\gamma}. $$
Similarly, we have
$$ \frac{1 - 2\gamma}{1 - \gamma} \leq \beta_{xy} \leq \frac{1 - \gamma}{1 - 2\gamma}. $$
Therefore,
$$ 1 - \frac{2\gamma}{3(1 - \gamma)} \leq \mathbf{b}_e \leq 1 + \frac{2\gamma}{3(1 - 2\gamma)}. $$
Since $\frac{2\gamma}{3(1 - \gamma)} \leq \frac{2\gamma}{3(1 - 2\gamma)}$, this implies that
$$ |1-\mathbf{b}_e| \leq \frac{2\gamma}{3(1 - 2\gamma)}. $$
This completes the proof.
\end{proof}

\begin{Lemma}\label{lem:M2bound}
For $0 \leq \gamma \leq 1/4$, the vector $\mathbf{b}^\star \in \mathbb{R}^{E(G)}$ satisfies $ M_2(\mathbf{b}^\star) \leq \frac{2\gamma}{3(1 - 2\gamma)}$.
\end{Lemma}

\begin{proof}
Without loss of generality, fix an arbitrary vertex $x \in X$. It suffices to show that $\frac{|S_{\mathbf{b}^\star}(x)|}{n} \leq \frac{2\gamma}{3(1 - 2\gamma)}$. Since $G$ is triangle-divisible, write $d_x := d(x,Y) = d(x,Z)$. Let
$$ C_x := \sum_{yz : xyz \in \mathcal{T}(G)} \frac{1}{t_{yz}} = \sum_{y \in N(x,Y)} \sum_{z : xyz \in \mathcal{T}(G)} \frac{1}{t_{yz}}. $$
Observe that
\begin{align*}
\sum_{y \in N(x,Y)} \sum_{z : xyz \in \mathcal{T}(G)} \frac{1}{t_{xy}} &= \sum_{y \in N(x,Y)} 1 = d_x, \\
\sum_{y \in N(x,Y)} \sum_{z : xyz \in \mathcal{T}(G)} \frac{1}{t_{xz}} &= \sum_{z \in N(x,Z)} \frac{1}{t_{xz}} \left( \sum_{y \in N(x,Y) \cap N(z,Y)} 1 \right) = \sum_{z \in N(x,Z)} \frac{1}{t_{xz}} ( t_{xz} ) = d_x, \\
\sum_{y \in N(x,Y)} \sum_{z : xyz \in \mathcal{T}(G)} \frac{1}{t_{yz}} &= C_x.
\end{align*}
Combining these, since $\mathbf{b}= A \mathbf{a}$, we have
\begin{align*}
\sum_{y \in N(x,Y)} \mathbf{b}_{xy} &= \sum_{y \in N(x,Y)} \sum_{z : xyz \in \mathcal{T}(G)} \mathbf{a}_{xyz} 
= \frac{1}{3} \sum_{y \in N(x,Y)} \sum_{z : xyz \in \mathcal{T}(G)} \left( \frac{1}{t_{xy}} + \frac{1}{t_{xz}} + \frac{1}{t_{yz}} \right) 
= \frac{2}{3} d_x + \frac{1}{3}C_x.
\end{align*}
Thus,
$$ S_{\mathbf{b}^\star}(x) = \sum_{y \in N(x,Y)} \mathbf{b}^\star_{xy} = \sum_{y \in N(x,Y)} ( 1 - \mathbf{b}_{xy} ) = d_x - \sum_{y \in N(x,Y)} \mathbf{b}_{xy} = \frac{1}{3} ( d_x - C_x). $$

Let $T_x$ denote the number of triangles containing $x$ in $G$. Since $(1-2\gamma) n \leq t_{yz} \leq n$ for any $yz \in E(Y,Z)$, we have
$$ \frac{T_x}{n} \leq C_x \leq \frac{T_x}{(1-2\gamma) n}.$$
On one hand,
\begin{align*}
T_x & = \sum_{y \in N(x,Y)} t_{xy} = \sum_{y \in N(x,Y)} | N(y,Z) \cap N(x,Z) | \\
& = \sum_{y \in N(x,Y)} (| N(y,Z) | + | N(x,Z) | - | N(y,Z) \cup N(x,Z) |) \\
&\geq \sum_{y \in N(x,Y)} ((1-\gamma) n + d_x - n)= \sum_{y \in N(x,Y)} (d_x - \gamma n) = d_x (d_x - \gamma n).
\end{align*}
Consequently, since $C_x \geq \frac{T_x}{n} \geq \frac{d_x(d_x - \gamma n)}{n}$ and $(1-\gamma)n \leq d_x \leq n$, we have
$$ d_x - C_x \le d_x - \frac{d_x(d_x - \gamma n)}{n} = \frac{d_x (n - d_x)}{n} + \gamma d_x \leq \frac{n \cdot \gamma n}{n} + \gamma n = 2\gamma n. $$
On the other hand, Clearly, $T_x \leq d_x^2$. Since $C_x \leq \frac{T_x}{(1-2\gamma)n} \leq \frac{d_x^2}{(1-2\gamma)n}$, we have
$$ C_x - d_x \leq \frac{d_x^2}{(1-2\gamma)n} - d_x = \frac{d_x (d_x - n + 2\gamma n)}{(1-2\gamma)n} \leq \frac{n \cdot 2\gamma n}{(1-2\gamma)n} = \frac{2\gamma n}{1-2\gamma}. $$
Since $2\gamma n \leq \frac{2\gamma n}{1-2\gamma}$, it follows that $|d_x - C_x| \leq \frac{2\gamma n}{1-2\gamma}$. Hence,
$$ \frac{|S_{\mathbf{b}^\star}(x)|}{n} = \frac{|d_x - C_x|}{3n} \leq \frac{1}{3n} \cdot \frac{2\gamma n}{1-2\gamma} = \frac{2\gamma}{3(1 - 2\gamma)}. $$
This completes the proof.
\end{proof}

\begin{Lemma}\label{lem:M3bound}
The vector $\mathbf{b}^\star \in \mathbb{R}^{E(G)}$ satisfies $ M_3(\mathbf{b}^\star) = 0$.
\end{Lemma}

\begin{proof}
Since $G$ is triangle-divisible, let $m:=|E(X, Y)| = |E(X, Z)| = |E(Y, Z)|$.
Observe that
$$ \sum_{xyz \in \mathcal{T}(G)} \frac{1}{t_{xy}} = \sum_{xy \in E(X,Y)} \sum_{z \in Z : xyz \in \mathcal{T}(G)} \frac{1}{t_{xy}} = \sum_{xy \in E(X,Y)} 1 = m. $$
Similarly, we have $\sum_{xyz \in \mathcal{T}(G)} \frac{1}{t_{xz}} = \sum_{xyz \in \mathcal{T}(G)} \frac{1}{t_{yz}} = m$.
Therefore,
\begin{align*}
\sum_{xy \in E(X,Y)} \mathbf{b}_{xy} &= \sum_{xy \in E(X,Y)} \sum_{z \in Z : xyz \in \mathcal{T}(G)} \mathbf{a}_{xyz} = \sum_{xyz \in \mathcal{T}(G)} \mathbf{a}_{xyz} 
= \sum_{xyz \in \mathcal{T}(G)} \frac{1}{3} \left( \frac{1}{t_{xy}} + \frac{1}{t_{xz}} + \frac{1}{t_{yz}} \right) \\
&= \frac{1}{3} \left( \sum_{xyz \in \mathcal{T}(G)} \frac{1}{t_{xy}} + \sum_{xyz \in \mathcal{T}(G)} \frac{1}{t_{xz}} + \sum_{xyz \in \mathcal{T}(G)} \frac{1}{t_{yz}} \right) 
= \frac{1}{3} (m+m+m) = m.
\end{align*}
Consequently, we have
$$ S(\mathbf{b}^\star) = \sum_{xy \in E(X,Y)} (1- \mathbf{b}_{xy}) = \left( \sum_{xy \in E(X,Y)} 1 \right) - \left( \sum_{xy \in E(X,Y)} \mathbf{b}_{xy} \right) = m-m = 0. $$
Hence, $M_3(\mathbf{b}^\star) = \frac{|S(\mathbf{b}^\star)|}{n^2} = 0$.
\end{proof}

Now we can give the proof of Lemma \ref{lem:MMbound}.

\begin{proof}[\bf Proof of Lemma \ref{lem:MMbound}]
Combine Lemmas \ref{lem:M1bound}, \ref{lem:M2bound}, and \ref{lem:M3bound} to complete the proof.
\end{proof}

The following corollary, which follows immediately from the definition of $M(\mathbf{b}^\star)$ and Lemmas \ref{lem:M1bound}, \ref{lem:M2bound}, and~\ref{lem:M3bound}, will be used to prove Lemma~\ref{lem:estimation} in Section~\ref{sec:estimation}.

\begin{Corollary}\label{cor:Mbound}
For $0 \leq \gamma \leq 1/4$, the vector $\mathbf{b}^\star \in \mathbb{R}^{E(G)}$ satisfies $ M(\mathbf{b}^\star) \leq \frac{10\gamma}{3(1-2\gamma)}$.
\end{Corollary}

\section{Proof of Lemma \ref{lem:estimation}}\label{sec:estimation}

This section proves Lemma~\ref{lem:estimation}. We first establish auxiliary Lemmas \ref{lem:M1Delta}, \ref{lem:M2Delta}, and \ref{lem:M3Delta} using Lemmas \ref{lem:InvBoundNonTriangle} and \ref{lem:DeltaCxy}; Lemma \ref{lem:estimation} then follows from these three lemmas. Recall that, whenever necessary, we extend a vector $\mathbf{c} \in \mathbb{R}^{E(G)}$ to $\mathbb{R}^{E(K_{X,Y,Z})}$ by setting its entries outside $E(G)$ to zero.

\begin{Lemma}\label{lem:InvBoundNonTriangle}
For any good vector $\mathbf{c} \in \mathbb{R}^{E(G)}$ and any triple $xyz \in X \times Y \times Z$ such that $xyz \notin \mathcal{T}(G)$, we have $ |\Inv(\mathbf{c})_{xyz}| \leq \frac{M(\mathbf{c})}{n}$.
\end{Lemma}

\begin{proof}
Suppose that a triple $xyz \in X \times Y \times Z$ does not span a triangle in $G$. Then, at least one of the pairs $xy, xz, yz$ is not an edge in $G$. Thus, at most two of the terms $\mathbf{c}_{xy}, \mathbf{c}_{xz}, \mathbf{c}_{yz}$ can be non-zero. Consequently,
$$ \left| \frac{\mathbf{c}_{xy} + \mathbf{c}_{xz} + \mathbf{c}_{yz}}{n} \right| \le \frac{2M_1(\mathbf{c})}{n}. $$
By the definition of $M_2(\mathbf{c})$, we have
$$ \left| \frac{S_\mathbf{c}(x) + S_\mathbf{c}(y) + S_\mathbf{c}(z)}{n^2} \right| \le \frac{3M_2(\mathbf{c})}{n}. $$
Recalling that $\frac{|S(\mathbf{c})|}{n^3} = \frac{M_3(\mathbf{c})}{n}$ and $M(\mathbf{c}) = 2M_1(\mathbf{c}) + 3M_2(\mathbf{c}) + M_3(\mathbf{c})$, we obtain
$$ |\Inv(\mathbf{c})_{xyz}| \leq \frac{2M_1(\mathbf{c})}{n} + \frac{3M_2(\mathbf{c})}{n} + \frac{M_3(\mathbf{c})}{n} = \frac{M(\mathbf{c})}{n}. $$
This completes the proof.
\end{proof}

For convenience, given a good vector $\mathbf{c} \in \mathbb{R}^{E(G)}$, we set 
$$\Delta\mathbf{c} := \mathbf{c} - A \widetilde{\Inv}(\mathbf{c}) \in \mathbb{R}^{E(G)}.$$ 
By Lemmas \ref{lem:repairable1} and \ref{lem:repairable2}, $\Delta\mathbf{c}$ is also a good vector.

\begin{Lemma}\label{lem:DeltaCxy}
For any good vector $\mathbf{c} \in \mathbb{R}^{E(G)}$ and any edge $xy \in E(X,Y)$, 
$$ \Delta\mathbf{c}_{xy} = \sum_{z \in Z \setminus (N(x,Z) \cap N(y,Z))} \Inv(\mathbf{c})_{xyz} = \sum_{z \in Z : xyz \notin \mathcal{T}(G)} \Inv(\mathbf{c})_{xyz}. $$
\end{Lemma}

\begin{proof}
By Lemma \ref{lem:Inv}, we have $ \sum_{z \in Z} \Inv(\mathbf{c})_{xyz} = \mathbf{c}_{xy}$. Therefore,
\begin{align*}
\Delta\mathbf{c}_{xy}
&= (\mathbf{c} - A \widetilde{\Inv}(\mathbf{c}))_{xy} = \mathbf{c}_{xy} - \sum_{T \in \mathcal{T}(G) : xy \in T} \widetilde{\Inv}(\mathbf{c})_T \\
&= \mathbf{c}_{xy} - \sum_{T \in \mathcal{T}(G) : xy \in T} \Inv(\mathbf{c})_T = \mathbf{c}_{xy} - \sum_{z \in N(x,Z) \cap N(y,Z)} \Inv(\mathbf{c})_{xyz} \\
&= \sum_{z \in Z} \Inv(\mathbf{c})_{xyz} - \sum_{z \in N(x,Z) \cap N(y,Z)} \Inv(\mathbf{c})_{xyz} \\
&= \sum_{z \in Z \setminus (N(x,Z) \cap N(y,Z))} \Inv(\mathbf{c})_{xyz} = \sum_{z \in Z : xyz \notin \mathcal{T}(G)} \Inv(\mathbf{c})_{xyz}.
\end{align*}
This completes the proof.
\end{proof}

\begin{Lemma}\label{lem:M1Delta}
For $0 \leq \gamma \leq 1/4$ and any good vector $\mathbf{c} \in \mathbb{R}^{E(G)}$, $M_1(\Delta\mathbf{c}) \leq 2\gamma M(\mathbf{c})$.
\end{Lemma}

\begin{proof}
Without loss of generality, fix an arbitrary edge $xy \in E(X,Y)$. It suffices to show that $| \Delta\mathbf{c}_{xy} | \leq 2\gamma M(\mathbf{c})$. By Lemma \ref{lem:DeltaCxy}, 
$ \Delta\mathbf{c}_{xy} = \sum_{z \in Z \setminus (N(x,Z) \cap N(y,Z))} \Inv(\mathbf{c})_{xyz}. $
Since $\hat{\delta}(G) \geq (1-\gamma)n$, we have $| \{ z \in Z \setminus (N(x,Z) \cap N(y,Z)) \} | \le 2\gamma n$.
Combining this with Lemma \ref{lem:InvBoundNonTriangle}, we have
$$ |\Delta\mathbf{c}_{xy}| \leq \sum_{z \in Z \setminus (N(x,Z) \cap N(y,Z))} |\Inv(\mathbf{c})_{xyz}| \leq 2\gamma n \cdot \frac{M(\mathbf{c})}{n} = 2\gamma M(\mathbf{c}).$$
This completes the proof.
\end{proof}

\begin{Lemma}\label{lem:M2Delta}
For $0 \leq \gamma \leq 1/4$ and any good vector $\mathbf{c} \in \mathbb{R}^{E(G)}$, $M_2(\Delta\mathbf{c}) \leq (\gamma + \gamma^2)M(\mathbf{c})$.
\end{Lemma}

\begin{proof}
Without loss of generality, fix an arbitrary vertex $x \in X$. It suffices to show that $|S_{\Delta\mathbf{c}}(x)| \leq n (\gamma + \gamma^2)M(\mathbf{c})$. By Lemma \ref{lem:DeltaCxy}, we have
\begin{align*}
S_{\Delta\mathbf{c}}(x)
&= \sum_{y \in Y} \Delta\mathbf{c}_{xy} = \sum_{y \in N(x,Y)} \Delta\mathbf{c}_{xy} = \sum_{y \in N(x,Y)} \left( \sum_{z \in Z : xyz \notin \mathcal{T}(G)} \Inv(\mathbf{c})_{xyz} \right) \\
&= \sum_{y \in N(x,Y)} \left( \sum_{z \in Z \setminus N(x,Z)} \Inv(\mathbf{c})_{xyz} + \sum_{z \in N(x,Z) \setminus N(y,Z)} \Inv(\mathbf{c})_{xyz} \right) \\
&= \sum_{z \in Z \setminus N(x,Z)} \sum_{y \in N(x,Y)} \Inv(\mathbf{c})_{xyz} + \sum_{y \in N(x,Y)} \sum_{z \in N(x,Z) \setminus N(y,Z)} \Inv(\mathbf{c})_{xyz}.
\end{align*}
We bound these two terms separately.

First term: For $z \in Z \setminus N(x,Z)$, since $\mathbf{c} \in \mathbb{R}^{E(G)}$, we have $\mathbf{c}_{xz} = 0$. By Lemma \ref{lem:Inv}, $\sum_{y \in Y} \Inv(\mathbf{c})_{xyz} = \mathbf{c}_{xz} = 0$. Therefore,
$$ \sum_{y \in N(x,Y)} \Inv(\mathbf{c})_{xyz} = - \sum_{y \in Y \setminus N(x,Y)} \Inv(\mathbf{c})_{xyz}. $$
Combining this with Lemma \ref{lem:InvBoundNonTriangle}, we have
\begin{align*}
\left| \sum_{z \in Z \setminus N(x,Z)} \sum_{y \in N(x,Y)} \Inv(\mathbf{c})_{xyz} \right| &\leq \sum_{z \in Z \setminus N(x,Z)} \left| \sum_{y \in Y \setminus N(x,Y)} \Inv(\mathbf{c})_{xyz} \right| \\
&\leq \gamma n \cdot \gamma n \cdot \frac{M(\mathbf{c})}{n} = n \gamma^2 M(\mathbf{c}).
\end{align*}

Second term: For any given $y \in Y$, we have $|N(x,Z) \setminus N(y,Z)| \leq |Z \setminus N(y,Z)| \leq \gamma n$. Combining this with Lemma \ref{lem:InvBoundNonTriangle}, we have
\begin{align*}
\left| \sum_{y \in N(x,Y)} \sum_{z \in N(x,Z) \setminus N(y,Z)} \Inv(\mathbf{c})_{xyz} \right| &\leq \sum_{y \in N(x,Y)} \left| \sum_{z \in N(x,Z) \setminus N(y,Z)} \Inv(\mathbf{c})_{xyz} \right| \\
&\leq \sum_{y \in N(x,Y)} \gamma n \cdot \frac{M(\mathbf{c})}{n} \leq n \cdot \gamma n \cdot \frac{M(\mathbf{c})}{n} = n \gamma M(\mathbf{c}).
\end{align*}

Therefore, $ | S_{\Delta\mathbf{c}}(x) | \leq n \gamma^2 M(\mathbf{c}) + n \gamma M(\mathbf{c}) = n (\gamma + \gamma^2) M(\mathbf{c})$.
\end{proof}

\begin{Lemma}\label{lem:M3Delta}
For $0 \leq \gamma \leq 1/4$ and any good vector $\mathbf{c} \in \mathbb{R}^{E(G)}$, $M_3(\Delta\mathbf{c}) \leq 3\gamma^2 M(\mathbf{c})$.
\end{Lemma}

\begin{proof}
Without loss of generality, consider the sum $\sum_{xy \in E(X,Y)} \Delta\mathbf{c}_{xy}$. By Lemma \ref{lem:DeltaCxy}, 
$$ \sum_{xy \in E(X,Y)} \Delta\mathbf{c}_{xy} = \sum_{xy \in E(X,Y)} \sum_{z \in Z \setminus (N(x,Z) \cap N(y,Z))} \Inv(\mathbf{c})_{xyz}=S_1+S_2, $$
where
$$ S_1 = \sum_{xy \in E(X,Y)} \sum_{z \in Z \setminus N(x,Z)} \Inv(\mathbf{c})_{xyz} \quad \text{and} \quad S_2 = \sum_{xy \in E(X,Y)} \sum_{z \in N(x,Z) \setminus N(y,Z)} \Inv(\mathbf{c})_{xyz}. $$

To evaluate $|S_1|$, let $z \in Z \setminus N(x,Z)$. Then $xz \notin E(G)$. Since $\mathbf{c} \in \mathbb{R}^{E(G)}$, it follows that $\mathbf{c}_{xz} = 0$. By Lemma \ref{lem:Inv}, we have 
$ \sum_{y \in Y} \Inv(\mathbf{c})_{xyz} = \mathbf{c}_{xz} = 0$.
Thus,
$$ \sum_{y \in N(x,Y)} \Inv(\mathbf{c})_{xyz} = - \sum_{y \in Y \setminus N(x,Y)} \Inv(\mathbf{c})_{xyz}. $$
Combining this with Lemma \ref{lem:InvBoundNonTriangle}, we deduce
\begin{align*}
|S_1| &= \left| \sum_{xy \in E(X,Y)} \sum_{z \in Z \setminus N(x,Z)} \Inv(\mathbf{c})_{xyz} \right| = \left| \sum_{xz \notin E(X,Z)} \sum_{y \in N(x,Y)} \Inv(\mathbf{c})_{xyz} \right| \\
&\leq \sum_{xz \notin E(X,Z)} \left| \sum_{y \in N(x,Y)} \Inv(\mathbf{c})_{xyz} \right| = \sum_{xz \notin E(X,Z)} \left| \sum_{y \in Y \setminus N(x,Y)} \Inv(\mathbf{c})_{xyz} \right| \\
&\leq \gamma n^2 \cdot \gamma n \cdot \frac{M(\mathbf{c})}{n} = \gamma^2 n^2 M(\mathbf{c}).
\end{align*}
To evaluate $|S_2|$, let $z \in N(x,Z) \setminus N(y,Z)$. Then $yz \notin E(G)$, and so $\mathbf{c}_{yz} = 0$. By Lemma \ref{lem:Inv}, we have $\sum_{x \in X} \Inv(\mathbf{c})_{xyz} = \mathbf{c}_{yz} = 0$. 
Thus,
$$ \sum_{x \in N(y,X) \cap N(z,X)} \Inv(\mathbf{c})_{xyz} = - \sum_{x \in X \setminus (N(y,X) \cap N(z,X))} \Inv(\mathbf{c})_{xyz}. $$
Since $|X \setminus (N(y,X) \cap N(z,X))| \leq 2\gamma n$ and $|(Y\times Z)\setminus E(Y,Z)| \leq \gamma n^2$, we apply Lemma \ref{lem:InvBoundNonTriangle} to obtain
\begin{align*}
|S_2| &= \left| \sum_{xy \in E(X,Y)} \sum_{z \in N(x,Z) \setminus N(y,Z)} \Inv(\mathbf{c})_{xyz} \right| = \left| \sum_{yz \notin E(Y,Z)} \sum_{x \in N(y,X) \cap N(z,X)} \Inv(\mathbf{c})_{xyz} \right| \\
&\leq \sum_{yz \notin E(Y,Z)} \left| \sum_{x \in N(y,X) \cap N(z,X)} \Inv(\mathbf{c})_{xyz} \right| = \sum_{yz \notin E(Y,Z)} \left| \sum_{x \in X \setminus (N(y,X) \cap N(z,X))} \Inv(\mathbf{c})_{xyz} \right| \\
&\leq \gamma n^2 \cdot 2\gamma n \cdot \frac{M(\mathbf{c})}{n} = 2\gamma^2 n^2 M(\mathbf{c}).
\end{align*}
Therefore, $| \sum_{xy \in E(X,Y)} \Delta\mathbf{c}_{xy}| \leq |S_1| + |S_2| \leq 3\gamma^2 n^2 M(\mathbf{c})$, which yields $M_3(\Delta\mathbf{c}) \leq 3\gamma^2 M(\mathbf{c})$.
\end{proof}

Recall that $M(\mathbf{c})=2M_1(\mathbf{c}) + 3M_2(\mathbf{c}) + M_3(\mathbf{c})$. Combining Lemmas \ref{lem:M1Delta}, \ref{lem:M2Delta}, and \ref{lem:M3Delta}, we have the following corollary.

\begin{Corollary}\label{cor:MDelta}
For $0 \leq \gamma \leq 1/4$ and any good vector $\mathbf{c} \in \mathbb{R}^{E(G)}$, $M(\Delta\mathbf{c}) \leq (7\gamma + 6\gamma^2)M(\mathbf{c})$.
\end{Corollary}

Now we are in a position to prove Lemma \ref{lem:estimation}.

\begin{proof}[\textbf{\textup{Proof of Lemma \ref{lem:estimation}.}}]
By the recursive definition of $\mathbf{b}^{(k)}$, we have $\mathbf{b}^{(k+1)} = \Delta\mathbf{b}^{(k)}$. It follows from Corollary \ref{cor:MDelta} that $M(\mathbf{b}^{(k+1)}) \leq (7\gamma + 6\gamma^2) M(\mathbf{b}^{(k)})$. Thus, $ M(\mathbf{b}^{(k)}) \leq (7\gamma + 6\gamma^2)^k M(\mathbf{b}^{(0)})$. Recall that $\mathbf{b}^{(0)} = \mathbf{b}^\star$. By Corollary \ref{cor:Mbound}, $M(\mathbf{b}^{(0)}) \leq \frac{10\gamma}{3(1-2\gamma)}$. Therefore,
$$ M(\mathbf{b}^{(k)}) \leq (7\gamma + 6\gamma^2)^k \frac{10\gamma}{3(1-2\gamma)}.$$ 
This proves (1). Furthermore, combining Lemmas \ref{lem:M1Delta}, \ref{lem:M2Delta}, and \ref{lem:M3Delta}, we have
$$ 3M_1(\mathbf{b}^{(k+1)}) + 3M_2(\mathbf{b}^{(k+1)}) + M_3(\mathbf{b}^{(k+1)}) \leq  (9\gamma + 6\gamma^2) M(\mathbf{b}^{(k)}) \leq (9\gamma + 6\gamma^2)(7\gamma + 6\gamma^2)^k \frac{10\gamma}{3(1-2\gamma)}. $$
This proves (2).
\end{proof}

\section{Concluding remarks}\label{sec:con}

Let $r \geq 4$. The existence of $r-2$ mutually orthogonal Latin squares is equivalent to the existence of an $r$-clique decomposition of a complete balanced $r$-partite graph (cf. \cite{BKLOT}). Montgomery \cite{Montgomery} investigated the partite minimum degree threshold for a balanced $r$-partite graph to admit a fractional $r$-clique decomposition, and thereby proved that, for sufficiently large $n$, any $r-2$ mutually orthogonal $\frac{1}{10^6 r^3}$-dense partial Latin squares of order $n$ can be extended to $r-2$ mutually orthogonal Latin squares of order $n$. Very recently, the partite minimum degree threshold for a balanced $r$-partite graph to admit a fractional $s$-clique decomposition has been studied in \cite{FLY}, where $r > s \geq 3$. It would be interesting to investigate whether the method developed in this paper can improve the known bounds for these thresholds.

Very recently, Delcourt and Postle \cite{DP26} made a breakthrough by proving that for sufficiently large $n$, every triangle-divisible graph on $n$ vertices with minimum degree at least $\frac{3}{4}n$ admits a triangle decomposition. Inspired by this result, one may naturally ask whether their idea can be employed to investigate the problem of triangle decomposition of tripartite graphs.

\end{document}